\documentclass{amsart}
\usepackage{amsmath}
\usepackage{graphicx}
\usepackage{latexsym}
\usepackage{color}
\newtheorem{thm}{Theorem}[section]
\newtheorem{lem}[thm]{Lemma}
\newtheorem{cor}[thm]{Corollary}
\newtheorem{prop}[thm]{Proposition}

\theoremstyle{definition}
\newtheorem{defn}[thm]{Definition}
\newtheorem{rem}[thm]{Remark}

\newcommand{\Z}{\mathbb{Z}}

\begin{document}

\title[Non-holomorphic Lefschetz fibrations with $(-1)$-sections]
{Non-holomorphic Lefschetz fibrations with $(-1)$-sections}

\author[N.~Hamada]{Noriyuki~Hamada}
\address{Graduate School of Mathematical Sciences, The University of Tokyo, 3-8-1 Komaba, Meguro-ku, Tokyo, 153-8914, Japan}
\email{nhamada@ms.u-tokyo.ac.jp}
\author[R.~Kobayashi]{Ryoma~Kobayashi}
\address{Department of General Education, Ishikawa National College of Technology, Tsubata, Ishikawa, 929-0392, Japan}
\email{kobayashi\_ryoma@ishikawa-nct.ac.jp}
\author[N.~Monden]{Naoyuki~Monden}
\address{Department of Engineering Science, Osaka Electro-Communication University, Hatsu-cho 18-8, Neyagawa, 572-8530, Japan}
\email{monden@osakac.ac.jp}

\begin{abstract}
We construct two types of non-holomorphic Lefschetz fibrations over $S^2$ with $(-1)$-sections ---hence, they are fiber sum indecomposable--- by giving the corresponding positive relators. 
One type of the two does not satisfy the slope inequality (a necessary condition for a fibration to be holomorphic) and has a simply-connected total space, and the other has a total space that cannot admit any complex structure in the first place.  
These give an alternative existence proof for non-holomorphic Lefschetz pencils without Donaldson's theorem.
\end{abstract}

\maketitle

\setcounter{secnumdepth}{2}
\setcounter{section}{0}


\section{Introduction}
The notion of Lefschetz fibrations in the smooth category was introduced by Moishezon \cite{Moishezon} from algebraic geometry to study complex surfaces from topological viewpoint.
It is therefore natural to ask how far smooth(symplectic) Lefschetz fibrations are from holomorphic ones. 
One approach to this question is to construct various non-holomorphic examples. 
Motivated by this, we give the following results. 
\begin{thm}\label{thm3}
For each $g\geq 3$, there is a genus-$g$ non-holomorphic Lefschetz fibration $X\to S^2$ with a $(-1)$-section and $\pi_1(X)=1$ such that it does not satisfy the ``slope inequality". 
\end{thm}
\begin{thm}\label{thm2}
For each $g\geq 4$, there is a family of genus-$g$ non-holomorphic Lefschetz fibrations $X_{\widehat{U}_n} \to S^2$ with two disjoint $(-1)$-sections (for each positive integer $n$) such that $X_{\widehat{U}_n}$ does not admit any complex structure with either orientation and is not homotopically equivalent to $X_{\widehat{U}_m}$ when $n\neq m$.
\end{thm}
Here, a non-holomorphic Lefschetz fibration means that it is not isomorphic to any holomorphic one. 
We would like to emphasize that we are able to give explicit monodromy factorizations of the above fibrations although we only give a procedure to get such factorizations without explicitly showing them. 
The rest of this section, we give some background on Theorem~\ref{thm3} and~\ref{thm2}.

\subsection{Lefschetz fibrations with a $(-1)$-section} \ 

The reason that we focus on Lefschetz fibrations that have $(-1)$-sections is that they play an important role as follows. 
Blowing up at the base loci of a genus-$g$ Lefschetz pencil yields a genus-$g$ Lefschetz fibration with $(-1)$-sections, and conversely, blowing down of $(-1)$-sections of a genus-$g$ Lefschetz fibration gives a genus-$g$ Lefschetz pencil. 
Furthermore, a closed $4$-manifold admits a symplectic structure if and only if it admits a Lefschetz pencil (Donaldson \cite{Do} proved the ``if" part, and the ``only if" part was shown in \cite{GS}). 
On the other hand, a Lefschetz fibration with a $(-1)$-section is fiber sum indecomposable (see \cite{St3},\cite{Sm4}); hence, such a fibration can be considered ``prime" with respect to the fiber sum operation. 
Therefore, as a corollary of Theorem~\ref{thm3} and~\ref{thm2}, we obtain the following result. 
\begin{cor}
	For arbitrary $g\geq 3$, there exists a genus-$g$ non-holomorphic Lefschetz pencil on a simply-connected 4-manifold. 
	For arbitrary $g\geq 4$, there exists infinitely many genus-$g$ non-holomorphic Lefschetz pencils on 4-manifolds that cannot admit any complex structure with either orientation. 
\end{cor}
\begin{rem}
	In~\cite{Baykur2}, Baykur constructed infinitely many non-holomorphic genus-$3$ Lefschetz pencils 
        with explicit monodromies. 
	The $4$-manifolds obtained as the total spaces are not simply connected and do not admit any complex structure with either orientation. 
\end{rem}
\begin{rem}
	Donaldson's construction of Lefschetz pencils on symplectic $4$-manifolds immediately implies the existence of non-holomorphic Lefschetz pencils since there are symplectic $4$-manifolds that cannot be complex.
	Yet this does not tell much about the genera of the resulting pencils.
	Our result shows the existence of non-holomorphic Lefschetz pencils for \textit{arbitrary genus $g \geq 3$}.
\end{rem}

\subsection{The slope inequality and simply-connected examples} \ 
The ``slope inequality" derives from the geography problem of relatively minimal holomorphic fibrations. 
Let us consider a relatively minimal genus-$g$ holomorphic fibration $\mathcal{F}:S\to C$ where $S$ and $C$ are a complex surface and a complex curve, respectively. 
In \cite{Xi}, Xiao defined a certain numerical invariant $\lambda_\mathcal{F}$, called the ``slope" of $\mathcal{F}$, determined by the signature and Euler characteristic of $S$, the genera of $C$ and a generic fiber. 
Then he showed that every relatively minimal genus-$g$ holomorphic fibration $f$ satisfies that $4-4/g\leq \lambda_\mathcal{F}$. 
We call this inequality the slope inequality. 

The notion of the slope can be extended for (smooth) Lefschetz fibrations as $\lambda_\mathcal{F}$ is determined by topological invariants (see Section~\ref{sec:nonholomorphicity}).
Besides, it turns out that the slope inequality holds for any hyperelliptic Lefschetz fibration, especially any genus-$2$ Lefschetz fibration. %
Hain conjectured that every Lefschetz fibration over $S^2$ satisfies the slope inequality as well (see \cite{ABKP},\cite{EN}). 
This conjecture in fact fails;
the third author gave examples violating the slope inequality (Theorem 3.1 of \cite{Mo}).
In particular, those examples are non-holomorphic by Xiao's result. 
However, we do not know if they are fiber sum indecomposable. 
Hence, we ask the following question: \textit{Is there a fiber sum indecomposable Lefschetz fibration violating the slope inequality?} 
Theorem~\ref{thm3} together with the above-mentioned work of \cite{St3} and \cite{Sm4} implies that the answer to this question is positive for any $g \geq 3$. 

Let us consider a genus-$g$ non-holomorphic Lefschetz fibration $X \to S^2$ with a $(-1)$-section such that $\pi_1(X)=1$. 
To the best of our knowledge, all known such fibrations with explicit monodromy factorizations are Fuller's example $(g=3)$\footnote{It was shown by Smith \cite{Sm3} that Fuller's example is non-holomorphic.} 
and Endo-Nagami's examples \cite{EN} $(g=3,4,5)$. 
Theorem~\ref{thm3} gives such examples with explicit monodromy factorizations for arbitrary $g\geq 3$. 
\begin{rem}
	We do not know whether the examples in \cite{Sm3}, \cite{EN} and Theorem~\ref{thm3} have non-complex total spaces or not. 
    On the other hand, Li \cite{Li} constructed non-holomorphic Lefschetz pencils (fibrations with a $(-1)$-section) on complex surfaces. 
    However, their genus are implicit. 
\end{rem}

\subsection{Lefschetz fibrations with non-complex total space} \ 

Many Lefschetz fibrations with explicit monodromies and non-complex total spaces have been constructed using the (twisted) fiber sum operation (see for instance \cite{Sm},\cite{OS},\cite{FS1},\cite{Kor},\cite{AO}\footnote{Baykur has informed us that the examples in \cite{AO} should be fiber sum decomposable from Ozbagci's talk in Turkey few years ago.},\cite{AM},\cite{BK}).
They are non-holomorphic, however, do not have any $(-1)$-section since they are decomposable. 
On the other hand, Stipsicz \cite{St3} and, independently, Smith \cite{Sm4} proved that there are infinitely many fiber sum indecomposable Lefschetz fibrations with non-complex total spaces. 
Since the construction of these fibrations is based on Donaldson's Theorem \cite{Do}, their monodromy factorizations are not explicitly given. 
Theorem~\ref{thm2} gives infinitely many fiber sum indecomposable Lefschetz fibrations with a explicit monodromy factorizations and non-complex total spaces for any $g \geq 4$.

The fundamental group of the total space $X_{\widehat{U}_n}$ of a genus-$g$ Lefschetz fibration in the family in Theorem~\ref{thm2} is $\pi_1(X_{\widehat{U}_n})=\mathbb{Z}\oplus \mathbb{Z}_n$. 
From the work of \cite{OS} (see also \cite{Baykur1}), the $4$-manifold $X_{\widehat{U}_n}$ does not carry any complex structure with either orientation. 
For $g\geq 22$, non-holomorphic Lefschetz fibrations with the same property of Theorem~\ref{thm2} were constructed in \cite{KM} based on the technique of this paper. 
Theorem~\ref{thm2} improves this result.

\begin{rem}
Non-holomorphic genus-$2$ Lefschetz fibrations with finite cyclic fundamental groups and without any $(-1)$-sections were constructed in \cite{AM} by rationally blowing down a twisted fiber sum of two copies of Matsumoto's fibration. 
However, we do not know whether these are decomposable or not. 
\end{rem}

\noindent \textit{Acknowledgments.} 
The third author was supported by Grant-in-Aid for Young Scientists (B) (No. 16K17601), Japan Society for the Promotion of Science. 
The authors would like to thank R. Inanc Baykur for comments on this paper and pointing out that there are more various examples of non-holomorphic Lefschetz fibrations than we mentioned in the first version of the manuscript. 
We are also grateful to Anar Akhmedov for comments.

\section{Preliminaries}
\subsection{Notations}\label{notation}

\

For convenience sake, we first fix the notation and the symbols for the curves which we use throughout the paper.
Let $\Sigma_g$ be the closed oriented surface of genus $g$ standardly embedded in the $3$-space and $a_1,b_1,a_2,b_2,\ldots,a_g,b_g$ be the standard generators of the fundamental group $\pi_1(\Sigma_g)$ of $\Sigma_g$ as shown in Figure~\ref{generator}. 
For loops $a$ and $b$ in $\pi_1(\Sigma_g)$, the product $ab$ means that we traverses first $a$ then $b$ as usual.
\begin{figure}[hbt]
 \centering 
     \includegraphics[width=6cm]{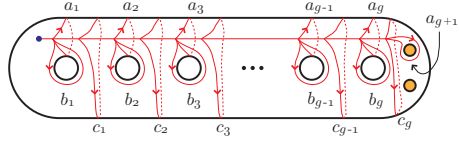}
     \caption{The standardly embedded $\Sigma_g$ with two indicated disks on the rightmost position and the generators $a_j,b_j$ of the fundamental group and loops $c_j$.}
     \label{generator}
\end{figure}

Let $c_1,c_2,\ldots,c_g$ and $a_{g+1}$ be the simple closed curves on $\Sigma_g$ as shown in Figure~\ref{generator}. 
Note that in $\pi_1(\Sigma_g)$, up to conjugation, we have 
\begin{align}\label{c}
c_i=b_i^{-1}\cdots b_2^{-1}b_1^{-1}(a_1b_1a_1^{-1})(a_2b_2a_2^{-1})\cdots (a_ib_ia_i^{-1}) 
\end{align}
for each $1\leq i\leq g$, in addition, $c_g=1$ and $a_{g+1}=1$. Then, the fundamental group $\pi_1(\Sigma_g)$ has the following presentation: 
\begin{align*}
\pi_1(\Sigma_g)=\langle a_1,b_1,a_2,b_2,\ldots, a_g,b_g \mid c_g\rangle.  
\end{align*}
\begin{figure}[hbt]
 \centering
      \includegraphics[width=9cm]{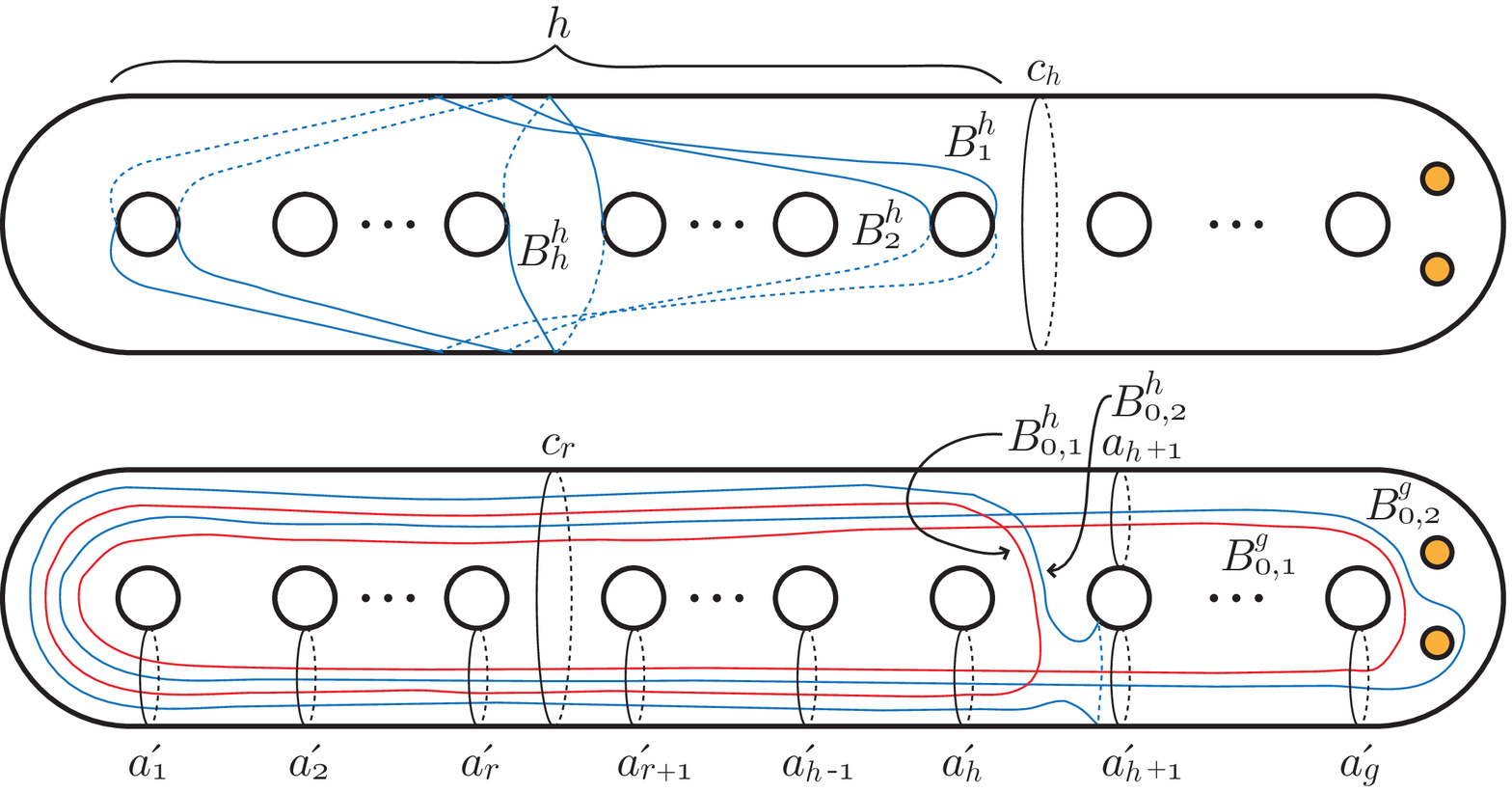}
      \caption{The curves $B_{0,1}^h, B_{0,2}^h, B_1^h, B_2^h, \ldots, B_h^h, a_1^\prime, a_2^\prime, \ldots, a_g^\prime$ for $h=2r$.}
      \label{Korkmaz-even}
 \end{figure}
\begin{figure}[hbt]
 \centering
      \includegraphics[width=9cm]{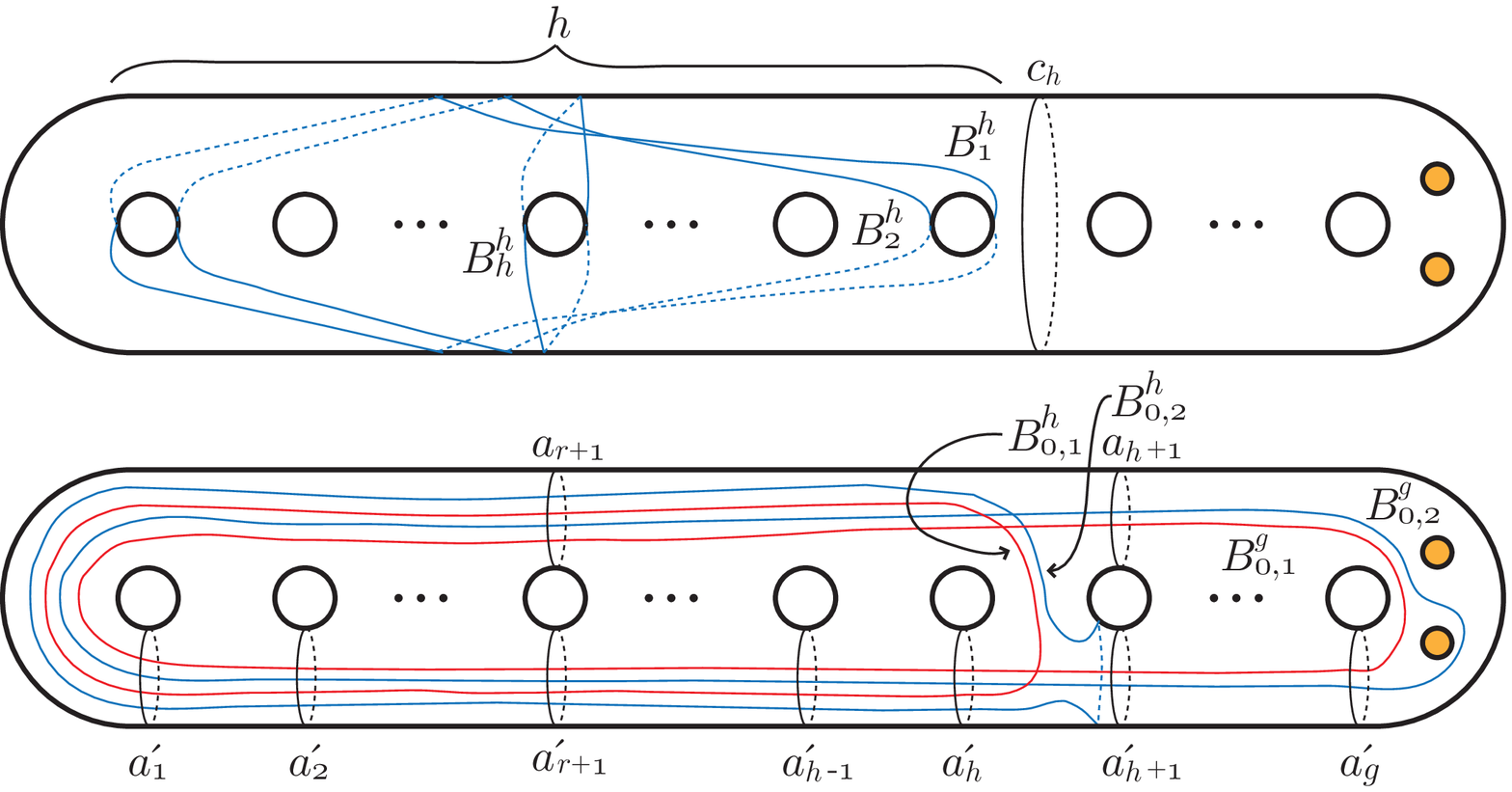}
      \caption{The curves $B_{0,1}^h, B_{0,2}^h, B_1^h, B_2^h, \ldots, B_h^h, a_1^\prime, a_2^\prime, \ldots,a_g^\prime $ for $h=2r+1$.}
      \label{Korkmaz-odd}
 \end{figure}

Let $B_{0,1}^h,B_{0,2}^h,B_1^h,B_2^h,\ldots, B_h^h$ $(h=1,2,\ldots, g)$ and $a_1^\prime,a_2^\prime,\ldots,a_g^\prime$ be the simple closed curves on $\Sigma_g$ as shown in Figure~\ref{Korkmaz-even} and~\ref{Korkmaz-odd}. 
Note that in $\pi_1(\Sigma_g)$, up to conjugation, we also have $a_h^\prime=c_ha_{h+1}$ for $1\leq h \leq g$. 

Suppose that $h=2r$. 
Then, it is easy to check that the following equalities hold in $\pi_1(\Sigma_g)$ up to conjugation: 
\begin{align}
&B_{0,1}^h=b_1b_2\cdots b_h, \ \ \ B_{0,2}^h=b_1b_2\cdots b_h \cdot c_ha_{h+1}, \ \ 1\leq h\leq g; \label{0sh}\\
&B_{2k-1}^h=a_k \cdot b_kb_{k+1} \cdots b_{h+1-k} \cdot c_{h+1-k}a_{h+1-k}, \ \ 1\leq k\leq r, \ 1\leq h \leq g; \label{2k-1}\\
&B_{2k}^h=a_k \cdot b_{k+1}b_{k+2} \cdots b_{h-k} \cdot c_{h-k}a_{h+1-k}, \ \ 1\leq k\leq r, \ 1\leq h \leq g. \label{2k}
\end{align}
In the case of $h=2r+1$, the same equalities (\ref{0sh}) and (\ref{2k}) hold without change, 
while the equality (\ref{2k-1}) hold for $1\leq k\leq r+1, 1\leq h \leq g$.

From now on, we use the same letter for a loop and its homotopy class by abuse of notation. 
Similarly, we use the same letter for a diffeomorphism and its isotopy class, or for a simple closed curve and its isotopy class. 
A simple loop and a simple closed curve are even denoted by the same letter. 
It will cause no confusion as it will be clear from the context which one we mean.

\subsection{Substitution technique}\label{MCG}

\

In this subsection, we introduce key techniques, called a \textit{substitution} and a \textit{partial conjugation}, for constructing a new word in mapping class groups from a word and a relator. 
We will utilize this technique to construct Lefschetz fibrations with $(-1)$-section in the later sections.

Let $\Sigma_g^b$ be a compact oriented surface of genus $g$ with $b$ boundary components. 
The \textit{mapping class group} $\Gamma_g^b$ of $\Sigma_g^b$ is the group of isotopy classes of orientation preserving self-diffeomorphisms of $\Sigma_g^b$, where all the maps involved are assumed to fix $\partial \Sigma_g^b$ pointwise. 
For simplicity, we write $\Sigma_g = \Sigma_g^0$ and $\Gamma_g^0=\Gamma_g$. 
For two elements $\phi_1$ and $\phi_2$ in $\Gamma_g^b$, the product $\phi_1\phi_2$ means that we first apply $\phi_2$ then $\phi_1$. 
We denote by $t_c$ be the right-handed \textit{Dehn twist} along a simple closed curve $c$ on $\Sigma_g^b$.

\begin{defn}\label{positive}\rm
Let $v_1,v_2,\ldots,v_n$ be simple closed curves on $\Sigma_g^b$. 
If $t_{v_n}^{\epsilon_n}\cdots t_{v_2}^{\epsilon_2}t_{v_1}^{\epsilon_1}=1$ in $\Gamma_g^b$, where $\epsilon_i=\pm 1$, then this factorization is called a \textit{relator}. 
In the special case where $\epsilon_i= 1$ for all $i$, namely, $t_{v_n}\cdots t_{v_2}t_{v_1}=1$ holds in $\Gamma_g$, then this factorization is called a \textit{positive relator}. 
\end{defn}

We introduce a key technique for constructing new product of right-handed Dehn twists in $\Gamma_g^b$ from old ones. 
\begin{defn}\label{substitution}\rm
Let $v_1,v_2,\ldots,v_k$ and $d_1,d_2,\ldots,d_l$ be simple closed curves on $\Sigma_g^b$ such that the following product, denoted by $R$, is a relator in $\Gamma_g^b$:
\begin{eqnarray*}
R:= \ t_{v_1}t_{v_2}\cdots t_{v_k}t_{d_l}^{-1}\cdots t_{d_2}^{-1}t_{d_1}^{-1}, 
\end{eqnarray*}
which equals the identity as a mapping class by definition.
If a mapping class $\phi$ in $\Gamma_g^b$ satisfies $\phi(d_i)=d_i$, then by the relation $t_{\phi(c)}=\phi t_c\phi^{-1}$, we obtain the following relator, denoted by $R^\phi$, in $\Gamma_g^b$: 
\begin{eqnarray*}
R^\phi = \ t_{\phi(v_1)}t_{\phi(v_2)}\cdots t_{\phi(v_k)}t_{d_l}^{-1}\cdots t_{d_2}^{-1}t_{d_1}^{-1}. 
\end{eqnarray*}
Let $W$ be a product of right-handed Dehn twists including $t_{d_1}t_{d_2}\cdots t_{d_l}$ as a subword: 
\begin{eqnarray*}
W=U\cdot t_{d_1}t_{d_2}\cdots t_{d_l} \cdot V, 
\end{eqnarray*}
where $U$ and $V$ are products of right-handed Dehn twists. 
Then, we get a new product of right-handed Dehn twists, denote by $W^\prime$, as follows: 
\begin{eqnarray*}
U \cdot R^\phi \cdot t_{d_1}t_{d_2}\cdots t_{d_l}\cdot V = U\cdot t_{\phi(v_1)}t_{\phi(v_2)}\cdots t_{\phi(v_k)} \cdot V =:W^\prime, 
\end{eqnarray*}
where the first equality means the equality as a mapping class. 
Then, $W^\prime$ is said to be obtained by applying a \textit{$R^\phi$-substitution} to $W$. 
\end{defn}
\begin{rem}
Fuller introduced the above operation for $\phi=1$. 
In the notation of Definition~\ref{substitution}, set $W_1=U\cdot t_{v_1}t_{v_2}\cdots t_{v_k} \cdot V$ and $W_2=U\cdot t_{\phi(v_1)}t_{\phi(v_2)}\cdots t_{\phi(v_k)} \cdot V$. 
Auroux \cite{Auroux3}, \cite{Auroux4} introduced the operation to obtain $W_2$ from $W_1$ is called a ``partial conjugation" by $\phi$. 
\end{rem}

\subsection{Relators in mapping class groups}\label{relator}

\

In this subsection, we introduce some well-known relators in mapping class groups, called 
the braid relator $B$, the lantern relator $L$, the chain relators $C_k, \overline{C}_k$ and certain relators $W_1^h, W_2^h$.

\begin{defn}[Braid relator]\label{braid}\rm
Let $\alpha$ and $\beta$ be simple closed curves on $\Sigma^b_g$. 
If the geometric intersection number of $\alpha$ and $\beta$ is equal to $0$ (resp. 1), then we have the \textit{braid relator} $B$:
\begin{align*}
B:=t_{\alpha}t_{\beta}t_{\alpha}^{-1}t_{\beta}^{-1} \ \ (\mathrm{resp.} \ \  B:=t_{\alpha}t_{\beta}t_{\alpha}t_{\beta}^{-1}t_{\alpha}^{-1}t_{\beta}^{-1}).
\end{align*}
\end{defn}

\begin{defn}[Lantern relator]\label{lantern}\rm
Let $\delta_{1}$, $\delta_{2}$, $\delta_{3}$ and $\delta_4$ be the four boundary curves of $\Sigma_0^4$ and let $\alpha$, $\beta$ and $\gamma$ be the interior curves as shown in Figure~\ref{LR}. 
Then, we have the \textit{lantern relator} $L$ in $\Gamma_0^4$: 
\begin{align*}
L:=t_{\alpha}t_{\beta}t_{\gamma}t_{\delta_4}^{-1}t_{\delta_3}^{-1}t_{\delta_2}^{-1}t_{\delta_1}^{-1}. 
\end{align*}
\end{defn}
\begin{figure}[hbt]
 \centering
      \includegraphics[width=2.5cm]{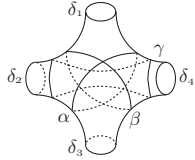}
      \caption{The curves $\delta_1, \delta_2, \delta_3,\delta_4$ and $\alpha, \beta, \gamma$.}
      \label{LR}
 \end{figure}
The lantern relator was discovered by Dehn (see \cite{De}) and was rediscovered by Johnson (see \cite{Jo}).

\begin{defn}[Chain relator]\label{chainr}\rm
Suppose $h \geq 1$. 
Let $\alpha_1, \alpha_2,\ldots,\alpha_{2h+1}$ be simple closed curves on an oriented surface such that $\alpha_i$ and $\alpha_{i+1}$ intersect transversally at exactly one point for $1 \leq i \leq 2h$ and that $\alpha_i$ and $\alpha_j$ are disjoint if $|i-j|\geq 2$. 
Then, a regular neighborhood of $\alpha_1\cup \alpha_2 \cup\cdots \cup \alpha_{2h}$ (resp. $\alpha_1\cup \alpha_2 \cup\cdots \cup \alpha_{2h+1}$) is a subsurface of genus $h$ with one boundary component (resp. two boundary components), say $d$ (resp. $d_1$ and $d_2$). 
We then have the \textit{even chain relator} $C_{2h}$ in $\Gamma_h^1$ and the \textit{odd chain relator} $C_{2h+1}$ in $\Gamma_h^2$:
\begin{align*}
&C_{2h}:=(t_{\alpha_1}t_{\alpha_2}\cdots t_{\alpha_{2h}})^{4h+2}t_d^{-1}, \\
&C_{2h+1}:=(t_{\alpha_1}t_{\alpha_2}\cdots t_{\alpha_{2h+1}})^{2h+2}t_{d_2}^{-1}t_{d_1}^{-1}.
\end{align*}
\end{defn}

\begin{defn}\label{MCK} \rm
Suppose $g \geq 2$. 
Let $\Sigma_g^2$ be the surface of genus $g$ with two boundary components obtained from $\Sigma_g$ by removing two disjoint open disks (cf. Figure~\ref{generator},~\ref{Korkmaz-even} and~\ref{Korkmaz-odd}). 
Let $a_{g+1}$ be one of the boundary curves of $\Sigma_g^2$ as shown in Figure~\ref{generator}, and let $a_{g+1}^\prime$ be the other boundary curve of $\Sigma_g^2$ defined by $a_{g+1}^\prime=c_ga_{g+1}$. 
We then have the following two relators $W_{1,h}$, $W_{2,h}$ in $\Gamma_g^2$ for each $h= 1,2,\ldots,g$: 
\begin{align*}
&W_{1,h}:=
  \left\{ \begin{array}{ll}
      \displaystyle ( t_{B_{0,1}^h} t_{B_1^h} t_{B_2^h} \cdots t_{B_{h-1}^h} t_{B_h^h} t_{c_r})^2 t_{c_h}^{-1} & \ \ (h=2r) \\[3mm]
      \displaystyle ( t_{B_{0,1}^h} t_{B_1^h} t_{B_2^h} \cdots t_{B_{h-1}^h} t_{B_h^h} t_{a_{r+1}}^2 t_{a_{r+1}^\prime}^2 )^2 t_{c_h}^{-1} & \ \ (h=2r+1),
      \end{array} \right.\\[3pt]
&W_{2,h}:=
  \left\{ \begin{array}{ll}
      \displaystyle ( t_{B_{0,2}^h} t_{B_1^h} t_{B_2^h} \cdots t_{B_{h-1}^h} t_{B_h^h} t_{c_r} )^2 t_{a_{h+1}}^{-1}t_{a^\prime_{h+1}}^{-1} & \ \ (h=2r) \\[3mm]
      \displaystyle ( t_{B_{0,2}^h} t_{B_1^h} t_{B_2^h} \cdots t_{B_{h-1}^h} t_{B_h^h} t_{a_{r+1}}^2 t_{a_{r+1}^\prime}^2 )^2 t_{a_{h+1}}^{-1}t_{a_{h+1}^\prime}^{-1} & \ \ (h=2r+1).
      \end{array} \right.
\end{align*}
\end{defn}

Note that in $\Gamma_g$, the relator $W_{2,g}$ is a positive relator. 
Matsumoto \cite{Ma} discovered this positive relator for $g=2$ , and Cadavid \cite{Ca} and independently Korkmaz \cite{Kor} generalized Matsumoto's relator to $g\geq 3$. 
$W_{1,g}$ was  shown to be a relator in $\Gamma_g^1$ by Ozbagci and Stipsicz \cite{OS2}. 
In \cite{Kor2}, Korkmaz claims that $W_{2,g}$ is a relator in $\Gamma_g^2$ without proof. 
Yet, we can show it by applying the same argument in Section 2 of \cite{Kor} (for example see Section 6 of \cite{KM}).

\section{Lefschetz fibrations}
\subsection{Basics on Lefschetz fibrations}\label{Lefschetz}

\

We recall the definition and basic properties of Lefschetz fibrations. 
More details can be found in \cite{GS}.

\begin{defn}\rm
Let $X$ be a closed, oriented smooth $4$-manifold. 
A smooth map $f : X \rightarrow S^2$ is a \textit{Lefschetz fibration} if for each critical point $p$ of $f$ and $f(p)$, 
there are complex local coordinate charts agreeing with the orientations of $X$ and $S^2$ on which $f$ is of the form $f(z_{1},z_{2})=z_{1}z_{2}$. 
\end{defn}

It follows that $f$ has finitely many critical points $C=\{p_1,p_2,\ldots,p_n\}$. 
We can assume that $f$ is injective on $C$ and relatively minimal (i.e. no fiber contains a sphere with self-intersection number $-1$). 
Each fiber which contains a critical point, called \textit{singular fiber}, is obtained by ``collapsing'' a simple closed curve in the prescribed regular fiber to a point. 
We call the simple closed curve in the regular fiber the \textit{vanishing cycle}. 
If the genus of the regular fiber of $f$ is equal to $g$, then we call $f$ the \textit{genus-$g$ Lefschetz fibration}.

The monodromy of the fibration around a singular fiber $f^{-1}(f(p_i))$ is given by a right-handed Dehn twist along the corresponding vanishing cycle, denoted by $v_i$. 
Once we fix an identification of $\Sigma_{g}$ with the fiber over a base point of $S^2$, we can characterize the Lefschetz fibration  $f:X\rightarrow S^2$ by its \textit{monodromy representation} $\pi_{1}(S^2-f(C))\rightarrow \Gamma_{g}$. 
Here, this map is indeed an anti-homomorphism. 
Let $\gamma_1,\gamma_2,\ldots,\gamma_n$ be an ordered system of generating loops for $\pi_{1}(S^2-f(C))$ 
such that each $\gamma_i$ encircles only $f(p_i)$ and $\gamma_1\gamma_2\cdots\gamma_n=1$ in $\pi_{1}(S^2-f(C))$. 
Thus, the monodromy of $f$ 
comprises a positive relator 
\begin{eqnarray*}
t_{v_n}\cdots t_{v_2}t_{v_1}=1 \ \ {\rm in} \ \Gamma_g. 
\end{eqnarray*}
Conversely, for any positive relator $P$ in $\Gamma_g$, one can construct a genus-$g$ Lefschetz fibration over $S^2$ whose monodromy is $P$. 
Therefore, we denote a genus-$g$ Lefschetz fibration associated with a positive relator $P$ in $\Gamma_g$ by $f_P:X_P \rightarrow S^2$.

Two Lefschetz fibrations $f_{P_i}:X_{P_i}\to S^2$ $(i=1, 2)$ are said to be \textit{isomorphic} if there exist orientation preserving diffeomorphisms $H:X_{P_1}\to X_{P_2}$ and $h:S^2\to S^2$ such that $f_{P_2}\circ H = h\circ f_{P_1}$. 
According to theorems of Kas \cite{Kas} and Matsumoto \cite{Ma}, if $g\geq 2$, then the isomorphism class of a Lefschetz fibration is determined by a positive relator modulo \textit{simultaneous conjugations} 
\begin{align*}
t_{v_n}\cdots t_{v_2}t_{v_1} \sim t_{\phi(v_n)}\cdots t_{\phi(v_2)}t_{\phi(v_1)} \ \ {\rm for \ any} \ \phi \in \Gamma_g
\end{align*}
and \textit{elementary transformations} 
\begin{align*}
&t_{v_n}\cdots t_{v_{i+2}}t_{v_{i+1}}t_{v_i}t_{v_{i-1}}t_{v_{i-2}}\cdots t_{v_1}& &\sim& &t_{v_n}\cdots t_{v_{i+2}}t_{v_i}t_{t_{v_i}^{-1}(v_{i+1})}t_{v_{i-1}}t_{v_{i-2}}\cdots t_{v_1},&\\
&t_{v_n}\cdots t_{v_{i+2}}t_{v_{i+1}}t_{v_i}t_{v_{i-1}}t_{v_{i-2}}\cdots t_{v_1}& &\sim& &t_{v_n}\cdots t_{v_{i+2}}t_{v_{i+1}}t_{t_{v_i}(v_{i-1})}t_{v_i}t_{v_{i-2}}\cdots t_{v_1}.&
\end{align*}
Therefore, if $P_2$ is obtained by applying a series of elementary transformations and simultaneous conjugations to $P_1$, then 
\begin{align}\label{isomorphic}
&\sigma(X_{P_1})=\sigma(X_{P_2})& &\mathrm{and}& &e(X_{P_1})=e(X_{P_2}),& 
\end{align}
where $\sigma(X)$ and $e(X)$ stand for the signature and the Euler characteristic of a $4$-manifold $X$, respectively.

\subsection{Sections of Lefschetz fibrations}
\begin{defn}\rm
Let $f:X\to S^2$ be a Lefschetz fibration. 
A map $\sigma:S^2\rightarrow X$ is called a $k$-\textit{section} of $f$ if it satisfies $f\circ \sigma={\rm id}_{S^2}$ and the self-intersection number $[\sigma(S^2)]^2=k$, where $[\sigma(S^2)]$ is the homology class in $H_2(X;\Z)$. 
\end{defn}

If the factorization $P=t_{v_n}\cdots t_{v_2}t_{v_1}(=1)$ lifts from $\Gamma_g$ to $\Gamma_g^1$ as 
\begin{align*}
  t_\delta ^k=t_{\widetilde{v}_n}\cdots t_{\widetilde{v}_2}t_{\widetilde{v}_1} \ \ \ (i.e. \ 1=t_{\widetilde{v}_n}\cdots t_{\widetilde{v}_2}t_{\widetilde{v}_1}t_\delta ^{-k}), 
\end{align*}
then the Lefschetz fibration $f_P$ has a $(-k)$-section. 
Here, $\delta$ is the boundary curve of $\Sigma_g^1$ and $t_{\widetilde{v}_i}$ is a Dehn twist mapped to $t_{v_i}$ under $\Gamma _g^1 \to \Gamma _g$. 
Conversely, if a genus-$g$ Lefschetz fibration admits a $(-k)$-section, we obtain a relator of the above type in $\Gamma_g^1$.  
A similar relator holds for $b$ disjoint sections (in which case one has to work in the mapping class group $\Gamma_g^b$).

A necessary condition for a Lefschetz fibration to admit a $(-1)$-section was shown independently by Stipsicz \cite{St3} and Smith \cite{Sm4}:
\begin{thm}[\cite{St3},\cite{Sm4}]\label{Indecomposable}
Let $g \geq 1$. 
If a genus-$g$ Lefschetz fibration $f:X\to S^2$ admits a $(-1)$-section, 
then $f$ is fiber sum indecomposable. 
\end{thm}
Here, we recall the definition of fiber sum. 
Let $f_i: X_i \to S^2$ be a genus-$g$ Lefschetz fibration for $i=1,2$, and let $D_i$ be an open disk on $S^2$ which does not contain any critical values. 
Then, the \textit{fiber sum} $f_1 \#_F f_2 : X_1\#_FX_2 \rightarrow S^2$ is obtained by gluing $X_1-f_1^{-1}(D_1)$ and $X_2-f_2^{-1}(D_2)$ along their boundaries via a fiber-preserving orientation-reversing diffeomorphism and extending $f_1$ and $f_2$ in a natural way. 
A Lefschetz fibration is said to be \textit{fiber sum indecomposable} if it cannot be decomposed as a fiber sum of two Lefschetz fibrations each of which has at least one singular point.

For a Lefschetz fibration over $S^2$ with a positive relator and a section, we can determine the fundamental group of $X$ as follows:

\begin{lem}[cf.\cite{GS}]\label{lem2}\rm
Let $P$ be a positive relator $P=t_{v_n}\cdots t_{v_2}t_{v_1}$ in $\Gamma_g$.
Suppose that the corresponding genus-$g$ Lefschetz fibration $f:X_P\to S^2$ admits a section $\sigma$. 
Then, the fundamental group $\pi_1(X)$ is isomorphic to the quotient of $\pi_1(\Sigma_g)$ by the normal subgroup generated by the vanishing cycles $v_1,v_2,\ldots,v_n$. 
\end{lem}

\subsection{Signatures of Lefschetz fibrations}

\

This subsection gives two results about the signatures of Lefschetz fibrations.

Let $\Delta_{g}$ be the \textit{hyperelliptic mapping class group} of genus $g$, i.e., 
the subgroup of $\Gamma_g$ consisting of those mapping classes commuting with the isotopy class of an involution $\iota$ shown in Figure~\ref{iota}. 
Note that $\Delta_g=\Gamma_g$ for $g=1,2$ and that $t_c$ is in $\Delta_g$ if and only if $\iota(c)=c$. 

 \begin{figure}[hbt]
 \centering
      \includegraphics[width=7.5cm]{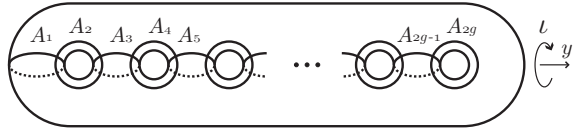}
      \caption{The involution $\iota$ of $\Sigma$ and the curves $A_1,A_2,\ldots,A_{2g}$ on $\Sigma_g$.}
      \label{iota}
 \end{figure}

A genus-$g$ Lefschetz fibration $f_P:X_\varrho\rightarrow S^2$ with a positive relator $P=t_{v_1}\cdots t_{v_n}$  is called \textit{hyperelliptic} if each $t_{v_i}$ in $\Delta_g$. 
To compute the signatures of Lefschetz fibrations, we present the Matsumoto-Endo's signature formula for hyperelliptic Lefschetz fibrations. 
\begin{thm}[\cite{Ma1},\cite{Ma} $(g=1,2)$,\cite{E} $(g\geq 3)$]\label{sign}
Let us consider a genus-$g$ hyperelliptic Lefschetz fibration $f_P:X_P\rightarrow S^2$ with $n$ nonseparating and $s=\Sigma_{h=1}^{[g/2]}s_h$ separating vanishing cycles, where $s_h$ is the number of separating vanishing cycles that separate $\Sigma_g$ into two surfaces, one of which has genus $h$. 
Then, we have
\begin{eqnarray*}
\sigma(X_P)=-\frac{g+1}{2g+1}n+\sum_{h=1}^{[g/2]}\left(\frac{4h(g-h)}{2g+1}-1\right)s_{h}.
\end{eqnarray*}
\end{thm}

By the works of Endo-Nagami \cite{EN}, we see the behavior of signatures of Lefschetz fibrations under a monodromy substitution as follows. 
\begin{prop}[\cite{EN}, Theorem 4.3, Definition 3.3, Lemma 3.5 and Proposition 3.9, 3.10 and 3.12]\label{EN}
Let $B$, $L$ and $C_{2h+1}$ be the braid relator, the lantern relator and the odd chain relator in Definition~\ref{braid},~\ref{lantern} and~\ref{chainr}, respectively. 
We assume that those relators are in $\Sigma_g$. 

Let $f_{P_i}:X_{P_i}\to S^2$ be a genus-$g$ Lefschetz fibration with a positive relator $P_i$ $(i=1,2)$. 
Suppose that $P_2$ is obtained by applying a $R^\phi$-substitution to $P_1$, where $\phi$ is a mapping class and $R$ is a relator in $\Gamma_g$. 
Then, the following holds. 
\begin{enumerate}
\setlength{\parskip}{0cm} 
\setlength{\itemsep}{0.05cm} 
\item[(a)] If $R=B$, then $\sigma(X_{P_2})=\sigma(X_{P_1})$, 

\item[(b)] If $R=L$, then $\sigma(X_{P_2})=\sigma(X_{P_1})+1$. Hence, if $R=L^{-1}$, then $\sigma(X_{P_2})=\sigma(X_{P_1})-1$, 

\item[(c)] Assume that both $d_1$ and $d_2$ are not null-homotopic in $\Sigma_g$. 
If $R=C_{2h+1}$, then $\sigma(X_{P_2})=\sigma(X_{P_1})+2h(h+2)$. Hence, if $R=C_{2h+1}^{-1}$, then $\sigma(X_{P_2})=\sigma(X_{P_1})-2h(h+2)$. 
\end{enumerate}
\end{prop}

\subsection{Non-holomorphicity of Lefschetz fibrations} \label{sec:nonholomorphicity}

\

\begin{defn}\rm
Let $f:X\to S^2$ be a Lefschetz fibration. 
$f$ is \textit{holomorphic} if there are complex structures on both $X$ and $S^2$ with holomorphic projection $f$. 
$f$ is \textit{non-holomorphic} if it is not isomorphic to any holomorphic Lefschetz fibration. 
\end{defn}

Suppose that $g\geq 2$. 
In order to prove Theorem~\ref{thm3} and~\ref{thm2}, we introduce two sufficient conditions for a Lefschetz fibration to be non-holomorphic. 

One comes from the result of Xiao \cite{Xi}. 
For an almost complex 4-manifold $X$, we set $K^2(X):=3\sigma(X)+2e(X)$ and $\chi_h(X):=(\sigma(X)+e(X))/4$. 
Xiao proved the following theorem, called the \textit{slope inequality}:
\begin{thm}[\cite{Xi}]\label{xiao}
Every relatively minimal holomorphic genus-$g$ fibration $f$ on a complex surface $X$ over a complex curve $C$ of genus $k\geq 0$ 
satisfies the inequality 
\begin{align*}
4-4/g\leq \lambda_f, 
\end{align*}
where $\lambda_f:=\dfrac{K^2(X)-8(g-1)(k-1)}{\chi_h(X)-(g-1)(k-1)}$. 
\end{thm}
As a consequence of of Theorem~\ref{xiao}, we have: 
\begin{prop}\label{xi}
If a genus-$g$ Lefschetz fibration $f:X\to S^2$ satisfies the slope inequality $\lambda_f<4-4/g$, then $f$ is non-holomorphic.
\end{prop}
The other comes from the result of Ozbagci and Stipsicz \cite{OS}.
The following theorem can be concluded from the proof of Theorem 1.3 in \cite{OS}: 
\begin{thm}\label{OSnoncomplex}
If a genus-$g$ Lefschetz fibration $f:X\to S^2$ satisfies $\pi_1(X)=\mathbb{Z}\oplus\mathbb{Z}_n$ for some positive integer $n$, then $X$ admits no complex structure with either orientation, so $f$ is non-holomorphic. 
\end{thm}

\section{Non-holomorphic Lefschetz fibration admitting a $(-1)$-section}\label{non-holo-simply}
In this section, we prove Theorem~\ref{thm3} 
\setcounter{section}{1}
\setcounter{thm}{0}
\begin{thm}
For each $g\geq 3$, there is a genus-$g$ non-holomorphic Lefschetz fibration $X\to S^2$ with a $(-1)$-section and $\pi_1(X)=1$ such that it does not satisfy the slope inequality. 
\end{thm}

To prove this, we need a lemma. 
Suppose $g\geq 3$. 
Let $\Sigma_g^1$ be the surface of genus $g$ with one boundary component obtained from $\Sigma_g$ by removing the open disk whose boundary curve is $a_{g+1}$ (cf. Figure~\ref{generator}). 
Let us consider $A_1,A_2,\ldots,A_{2g}$ be the simple closed curves on $\Sigma_g^1$ (cf. Figure~\ref{curveAi}) defined as follows: $A_1=a_1$, $A_2 = b_1$, $A_{2h-1}=a_{h-1}a_{h}^{-1}$ and $A_{2h}=b_h$ for $h=2, 3, \cdots, g$. 

 \begin{figure}[hbt]
 \centering
      \includegraphics[width=7.5cm]{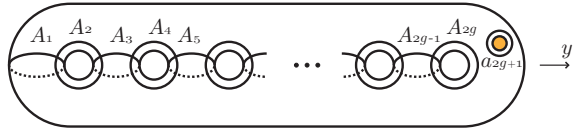}
      \caption{The curves $A_1,A_2,\ldots,A_{2g}$ on $\Sigma_g^1$.}
      \label{curveAi}
 \end{figure}

\setcounter{section}{4}
\setcounter{thm}{0}
\begin{lem}\label{lem1}
$(t_{A_1}t_{A_2}\cdots t_{A_{2g}})^{2g+1} = (t_{A_1}t_{A_2}\cdots t_{A_{2g-1}})^{2g}t_{A_{2g}}\cdots t_{A_2}t_{A_1}t_{A_1}t_{A_2}\cdots t_{A_{2g}}$.
\end{lem}
\begin{proof}
The proof follows from the braid relations $t_{A_i}t_{A_{i+1}}t_{A_i} = t_{A_{i+1}}t_{A_i}t_{A_{i+1}}$ and $t_{A_i}t_{A_j}=t_{A_j}t_{A_i}$ for $|i-j|>1$ (i.e. by applying $B$-substitutions to the left side).
\end{proof}

We now prove Theorem~\ref{thm3}. 
\begin{proof}[Proof of Theorem~\ref{thm3}]
Suppose $g\geq 3$. 
Let us consider the following chain relators $C_{2g}$ and $C_{2g+1}$:
\begin{align*}
&C_{2g}=(t_{A_1}t_{A_2}\cdots t_{A_{2g}})^{4g+2} t_{a_{g+1}}^{-1}, &C_{2g-1}=(t_{A_1}t_{A_2}\cdots t_{A_{2g-1}})^{2g} t_{a_g}^{-1}t_{a_g^\prime}^{-1},
\end{align*}
where $a_g$ and $a_g^\prime$ are the curves as shown in Figure~\ref{Korkmaz-even} and~\ref{Korkmaz-odd}. 
By Lemma~\ref{lem1} and the even chain relator $C_{2g}$, we obtain the following relator $C_{2g}^\prime$:
\begin{eqnarray*}
C_{2g}^\prime
&=&\{(t_{A_1}t_{A_2}\cdots t_{A_{2g-1}})^{2g} \cdot t_{A_{2g}}\cdots t_{A_2}t_{A_1}t_{A_1}t_{A_2}\cdots t_{A_{2g}}\}^2 t_{a_{g+1}}^{-1}. 
\end{eqnarray*}
By applying $C_{2g-1}^{-1}$-substitution to $C_{2g}$ twice, we get a new relator $H$ in $\Gamma_g^1$:
\begin{eqnarray*}
H=(t_{a_g}t_{a_g^\prime} \cdot t_{A_{2g}}\cdots t_{A_2}t_{A_1}t_{A_1}t_{A_2}\cdots t_{A_{2g}})^2 t_{a_{g+1}}^{-1}. 
\end{eqnarray*}

 \begin{figure}[hbt]
 \centering
      \includegraphics[width=5.5cm]{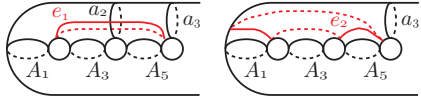}
      \caption{The curves that give a Lantern relator.}
      \label{Lcurves}
 \end{figure}

Consider the curves on $\Sigma_g^1$ in Figure~\ref{Lcurves}. 
Since $A_1,a_2,e_1,$ and $e_2$ are non-separating curves on the subsurface of genus $g-1$ with two boundary components $a_g$ and $a_g^\prime$, there are diffeomorphisms $\psi_1$, $\psi_2$ and $\psi_3$ in $\Gamma_g^1$ such that $\psi_1(A_1)=a_2$, $\psi_2(A_1)=e_1$, $\psi_3(A_1)=e_2$, and each $\psi_i$ is identical near $a_g$ and $a_g^\prime$. 
Then, we have the following relator $H^{\psi_1}$:
\begin{align*}
H^{\psi_1}=(t_{a_g}t_{a_g^\prime} \cdot t_{\psi_1(A_{2g})}\cdots t_{\psi_1(A_2)}t_{a_2}t_{a_2}t_{\psi_1(A_2)}\cdots t_{\psi_1(A_{2g})})^2 \cdot t_{a_{g+1}}^{-1}. 
\end{align*}
Applying $C_{2g-1}^{\psi_2}$- and $C_{2g-1}^{\psi_3}$-substitutions to $H^{\psi_1}$, we get a relator $H^\prime$:
\begin{align*}
H^\prime=&(t_{e_1}t_{\psi_2(A_2)}\cdots t_{\psi_2(A_{2g-1})})^{2g}t_{\psi_1(A_{2g})}\cdots t_{\psi_1(A_2)}t_{a_2}t_{a_2}t_{\psi_1(A_2)}\cdots t_{\psi_1(A_{2g})} \\
&\cdot (t_{e_2}t_{\psi_3(A_2)}\cdots t_{\psi_3(A_{2g-1})})^{2g}t_{\psi_1(A_{2g})}\cdots t_{\psi_1(A_2)}t_{a_2}t_{a_2}t_{\psi_1(A_2)}\cdots t_{\psi_1(A_{2g})} \cdot t_{a_{g+1}}^{-1}.
\end{align*}
Here, let us consider a word $t_c \cdot t_{v_1}t_{v_2}\cdots t_{v_k}$. 
By repeating elementary transformations to this word, we obtain the word $t_{t_c(v_1)}t_{t_c(v_2)}\cdots t_{t_c(v_k)} \cdot t_c$. 
Therefore, since $H^\prime$ is a positive relator including $t_{e_1}$, $t_{a_2}$ and $t_{e_2}$ in this order, we can put them together to the right side of the word to obtain a relator in the form 
\begin{align*}
H^{\prime\prime}=T \cdot t_{e_1}t_{a_2}t_{e_2} \cdot t_{a_{g+1}}^{-1},
\end{align*}
where $T$ is a product of $8g^2+4g-3$ right-handed Dehn twists. 
Let $L$ denote the lantern relator $L=t_{e_1}t_{a_2}t_{e_2}t_{A_1}^{-1}t_{a_3}^{-1}t_{A_5}^{-1}t_{A_3}^{-1}$. 
Finally, we do $L^{-1}$-substitution to $H^{\prime\prime}$, to obtain the following relator $I$ in $\Gamma_g^1$: 
\begin{align*}
I=T \cdot t_{A_3}t_{A_5}t_{a_3}t_{A_1} \cdot t_{a_{g+1}}^{-1}.
\end{align*}

The relator $I$ derives a positive relator $\widehat{I}$ in $\Gamma_g$. 
Thus, $\widehat{I}$ gives a genus-$g$ Lefschetz fibration $f_{\widehat{I}}:X_{\widehat{I}}\to S^2$ which admits a $(-1)$-section.

We see that a genus-$g$ Lefschetz fibration $f_{\widehat{I}}: X_{\widehat{I}} \to S^2$ has a $2g(4g+2)+1$ singular fibers. 
Hence, we have 
\begin{align*}
e(X_{\widehat{I}})=8g^2+5. 
\end{align*}
Here, note that $C_{2g}$ is a positive relator in $\Gamma_g$. 
This gives a genus-$g$ Lefschetz fibration $f_{C_{2g}} : X_{C_{2g}} \to S^2$ with $2g(4g+2)$ nonseparating singular fibers. 
In particular, this fibration is hyperelliptic since $\iota(A_i)=A_i$ for each $i=1,2,\ldots,2g$ (see Figure~\ref{iota}). 
Therefore, we have $\sigma(X_{C_{2g}})=-4g(g+1)$ by Theorem~\ref{sign}. 
Since $I$ is obtained from $C_{2g}$ by some $B$-substitutions, two $C_{2g-1}^{-1}$-substitutions, $C_{2g-1}^{\psi_2}$- and $C_{2g-1}^{\psi_3}$-substitutions, other several $B$-substitutions, and one $L^{-1}$-substitution, by (\ref{isomorphic}) and Proposition~\ref{EN}, we have 
\begin{align*}
\sigma(X_{\widehat{I}}) &= \sigma(X_{C_{2g}})-1 \\ 
                        &=-4g(g+1)-1. 
\end{align*}
This gives $\lambda_{f_{\widehat{I}}}=4-4/g-1/g^2<4-4/g$. 
By Proposition~\ref{xi}, this fibration is non-holomorphic. 

It is easy to check that $\widehat{I}$ includes the Dehn twist about the curve $t_{e_1}(\psi_1(A_i))$ for $1\leq i\leq 2g$. 
Since $f_{\widehat{I}}$ admits a section, by Lemma~\ref{lem2} we have 
\begin{align*}
\pi_1(X_{\widehat{I}}) \ \subset \ \pi_1(\Sigma_g)/ \langle t_{e_1}(\psi_1(A_1)), \cdots, t_{e_1}(\psi_1(A_{2g}))  \rangle. 
\end{align*}
On the other hand, it is easy to check that 
\begin{align*}
\pi_1(\Sigma_g)/ \langle t_{e_1}(\psi_1(A_1)), \cdots, t_{e_1}(\psi_1(A_{2g})) &= \pi_1(\Sigma_g)/ \langle A_1, \cdots, A_{2g}  \rangle \\
&=1,
\end{align*}
hence $\pi_1(X_{\widehat{I}}) =1$. 

This completes the proof of Theorem~\ref{thm3}. 
\end{proof}

\begin{rem}
      We do not provide a monodromy factorization of $f_{\widehat{I}}$ explicitly, however, we can obtain it by giving explicit $\psi_j(A_i)$ for $j=1,2,3$ and $i=1,2,\ldots,2g$. 
\end{rem}

\begin{rem}
      All vanishing cycles of the Lefschetz fibration $f_{\widehat{I}}$ are nonseparating since all curves of the lantern relator 
      employeed in the proof of Theorem~\ref{thm3} are nonseparating. 
      For $g\geq 3$, we can consider a lantern relator such that 
      six curves are nonseparating and one curve, denoted by $s_h$, is a separating, which separates $\Sigma_g^1$ into two subsurfaces $\Sigma_h^1$ and $\Sigma_{g-h}^2$ for $h\geq 2$. 
      Then, a similar argument to the proof of Theorem~\ref{thm3} gives a genus-$g$ Lefschetz fibration with a $(-1)$-section, the simply connected total space and the vanishing cycle $s_h$ $(h=2,3,\ldots,g-1)$ and violating the slope inequality. 
      Therefore, we can construct at least $g-1$ different genus-$g$ Lefschetz fibrations with the conditions in Theorem~\ref{thm3}. 
\end{rem}

\begin{rem}
      Miyachi and Shiga \cite{MS} produced genus-$g$ Lefschetz fibrations over $\Sigma_{2m}$ $(m\geq 1)$ which do not satisfy the slope inequality. 
\end{rem}

\section{Non-complex Lefschetz fibration admitting a $(-1)$-section}
In this section, we prove Theorem~\ref{thm2}. 
\setcounter{section}{1}
\setcounter{thm}{1}
\begin{thm}\label{thm2}
For each $g\geq 4$ and each positive integer $n$, there is a genus-$g$ non-holomorphic Lefschetz fibration $f_{\widehat{U}_n} : X_{\widehat{U}_n} \to S^2$ with two disjoint $(-1)$-sections such that $X_{\widehat{U}_n}$ does not admit any complex structure with either orientation. 
\end{thm}

We assume that $g \geq 4$ throughout this section. 
In order to prove Theorem~\ref{thm2}, we construct a relator $U_n$ in $\Gamma_g^2$ by applying substitutions to the relator $W_2^g$ in $\Gamma_g^2$, which gives the Lefschetz fibration $f_{\widehat{U}_n} : X_{\widehat{U}_n} \to S^2$.

Let $a_j,a_j^\prime, b_j$ and $c_j$ $(j=1,2,\ldots,g)$ be the simple closed curves on $\Sigma_g^2$ in Figure~\ref{generator}, \ref{Korkmaz-even} and~\ref{Korkmaz-odd}, and let $a_{g+1}$ and $a_{g+1}^\prime$ be the boundary curves of $\Sigma_g^2$ as before. 
The notation $[s]$ means the integer part of a real number $s$.

For a positive integer $n$, we define a map $\phi_n$ to be 
\begin{align*}
\phi_n=
  \left\{ \begin{array}{llll}
      \displaystyle t_{a_1}t_{a_2}\cdots t_{a_{k-1}}t_{a_k}^n \cdot t_{b_{k+2}}t_{b_{k+3}}\cdots t_{b_{2k}} & \ \ ([g/2]=2k) \\[3mm]
      \displaystyle t_{a_1}t_{a_2}\cdots t_{a_{k-1}}t_{a_k}^n \cdot t_{b_{k+3}}t_{b_{k+4}}\cdots t_{b_{2k+1}} & \ \ ([g/2]=2k+1).
      \end{array} \right.
\end{align*}
Note that $\phi_n(c_r)=c_r$ (resp. $\phi(a_{r+1})=a_{r+1}$ and $\phi(a_{r+1}^\prime)=a_{r+1}^\prime$) for $g=2r$ (resp. $g=2r+1$) 
and that $\phi_n(c_k)=c_k$ (resp. $\phi(a_{k+1})=a_{k+1}$ and $\phi(a_{k+1}^\prime)=a_{k+1}^\prime$) for $[g/2]=2k$ (resp. $[g/2]=2k+1$).

The relator $W_2^g$ in $\Gamma_g^2$ includes Dehn twist $t_{c_r}$ twice (resp. the product $t_{a_{r+1}}t_{a_{r+1}^\prime}$ of two Dehn twists  four times) if $g=2r$ (resp. $g=2r+1$). 
Therefore, we can apply $W_{1,h}$- and $W_{1,h}^{\phi_n}$- (resp. $W_{2,h}$- and $W_{2,h}^{\phi_n}$) substitutions to $W_2^g$ if $[g/2]=2t$ (resp. $[g/2]=2t+1$). 
Then, for even (resp. odd) $[g/2]$, we denote by 
\begin{align*}
U_n
\end{align*}
a relator which is obtained by applying once trivial and once $\phi_n$-twisted $W_1^h$- (resp. $W_2^h$-) substitutions to $W_2^g$. 
For the convenience of the reader we write the definition of the relator $U_n$ in detail. 
Let us consider the following word in $\Gamma_g^2$. 
\begin{align*}
&V_1:=
  \left\{ \begin{array}{ll}
      \displaystyle ( t_{B_{0,1}^{[g/2]}} t_{B_1^{[g/2]}} t_{B_2^{[g/2]}} \cdots t_{B_{[g/2]}^{[g/2]}} t_{c_t})^2  & \ \ ([g/2]=2t) \\[3mm]
      \displaystyle ( t_{B_{0,1}^{[g/2]}} t_{B_1^{[g/2]}} t_{B_2^{[g/2]}} \cdots t_{B_{[g/2]}^{[g/2]}} t_{a_{t+1}}^2 t_{a_{t+1}^\prime}^2 )^2  & \ \ ([g/2]=2t+1),
      \end{array} \right.\\[3pt]
&V_2:=
  \left\{ \begin{array}{ll}
      \displaystyle ( t_{B_{0,2}^{[g/2]}} t_{B_1^{[g/2]}} t_{B_2^{[g/2]}} \cdots t_{B_{[g/2]}^{[g/2]}} t_{c_t} )^2  & \ \ ([g/2]=2t) \\[3mm]
      \displaystyle ( t_{B_{0,2}^{[g/2]}} t_{B_1^{[g/2]}} t_{B_2^{[g/2]}} \cdots t_{B_{[g/2]}^{[g/2]}} t_{a_{t+1}}^2 t_{a_{t+1}^\prime}^2 )^2 & \ \ ([g/2]=2t+1).
      \end{array} \right.
\end{align*}
Note that $V_1=W_{1,[g/2]}t_{c_{[g/2]}}$ and $V_2= W_{2,[g/2]} t_{a_{[g/2]+1}^\prime}^{-1} t_{a_{[g/2]+1}}^{-1}$. 
Then, we can write $U_n$ as follows:
If $g=2r$, then 
\begin{align*}
U_n := (t_{B_{0,2}^g} t_{B_1^g} t_{B_2^g} \cdots t_{B_g^g} V_1) (t_{B_{0,2}^g} t_{B_1^g} t_{B_2^g} \cdots t_{B_g^g} V_1^{\phi_n}) t_{a_{g+1}}^{-1} t_{a_{g+1}^\prime}^{-1}, 
\end{align*}
and if $g=2r+1$, then 
\begin{align*}
U_n := (t_{B_{0,2}^g} t_{B_1^g} t_{B_2^g} \cdots t_{B_g^g} V_2 t_{a_{r+1}} t_{a_{r+1}^\prime}) (t_{B_{0,2}^g} t_{B_1^g} t_{B_2^g} \cdots t_{B_g^g} V_2^{\phi_n} t_{a_{r+1}} t_{a_{r+1}^\prime}) t_{a_{g+1}}^{-1} t_{a_{g+1}^\prime}^{-1}. 
\end{align*}

Since the relator $U_n$ in $\Gamma_g^2$ is a product of $t_{a_{g+1}}^{-1} t_{a_{g+1}^\prime}^{-1}$ and positive Dehn twists, it reduces to a positive relator of $\Gamma_g$, denoted by $\widehat{U}_n$. 
This gives a genus-$g$ Lefschetz fibration $f_{\widehat{U}_n}:X_{\widehat{U}_n}\rightarrow S^2$ with two disjoint $(-1)$-sections.

We now prove Theorem~\ref{thm2}. 
\begin{proof}[Proof of Theorem~\ref{thm2}]
It is sufficient to show that the fundamental group of $X_{\widehat{U}_n}$ is $\pi_1(X_{\widehat{U}_n}) =\mathbb{Z}\oplus\mathbb{Z}_n$ from Theorem~\ref{OSnoncomplex}.

Let $\mathcal{B}_s^h$ ($1 \leq h \leq g, s = 1,2$) be the following set of loops in $\pi_1(\Sigma_g)$: 
\begin{align*}
&\mathcal{B}_s^h:=
  \left\{ \begin{array}{ll}
      \displaystyle \{B_{0,s}^h,B_1^h,B_2^h,\ldots,B_h^h\} & \ \ (h=2r) \\[1mm]
      \displaystyle \{B_{0,s}^h,B_1^h,B_2^h,\ldots,B_h^h,a_{r+1},a_{r+1}^\prime\} & \ \ (h=2r+1).
      \end{array} \right.
\end{align*}
Note that $\mathcal{B}_1^g=\mathcal{B}_2^g$. 
We define the sets $\mathcal{C}$ and $\phi_n(\mathcal{C})$ of loops in $\pi_1(\Sigma_g)$ to be 
\begin{align*}
& \mathcal{C} :=
  \left\{ \begin{array}{ll}
      \displaystyle \mathcal{B}_s^{[g/2]}\cup \{c_t\} & \ ([g/2]=2t) \\[1mm]
      \displaystyle \mathcal{B}_s^{[g/2]} & \ ([g/2]=2t+1)
      \end{array} \right.,\ \ 
\phi_n(\mathcal{C}):=\{\phi_n(c) \mid c\in \mathcal{C}\},
\end{align*}
where $s=1$ if $g$ is even, and $s=2$ if $g$ is odd, and $\phi_n$ is the map as defined the above. 
Then, by Lemma~\ref{lem2}, we have
\begin{align*}
\pi_1(X_{\widehat{U}_n})=\pi_1(\Sigma_g)/\langle \mathcal{B}_1^g\cup \mathcal{C}_s^{[g/2]} \cup \phi_n(\mathcal{C}_s^{[g/2]})\rangle.
\end{align*}
Here, $\langle S \rangle$ means the normal closure of a subset $S$ of a group. 
For simplicity, we write 
\begin{align*}
&G_1:=\pi_1(\Sigma_g)/\langle \mathcal{B}_{1}^g \rangle, && G_2:=G_1/\langle \mathcal{C} \rangle &&\mathrm{and} && G_3:=G_2/\langle \phi_n(\mathcal{C}) \rangle.
\end{align*}
Note that  $G_3=\pi_1(X_{\widehat{U}_n})$.

First, we compute $G_1$. 
Suppose that $g=2r$. 
Since $c_g=1$, and the equalities (\ref{0sh}), (\ref{2k-1}) and (\ref{2k}) in Section~\ref{notation} with $h = g$ also become the identity in $G_1$, we obtain $a_k=a_{g+1-k}^{-1}$ for $1 \leq k \leq r$ after a routine computation. 
This gives 
\begin{align*}
&1=B_{2k-1}^g=b_k \cdot b_{k+1}b_{k+2}\cdots b_{g-k} \cdot b_{g+1-k}c_{g+1-k}, \ 1\leq k\leq r; \\
&1=B_{2k}^g=b_{k+1}b_{k+2}\cdots b_{g-k} \cdot c_{g-k}, \ 1\leq k\leq r.
\end{align*}
From these two equalities, we have $b_kc_{g-k}^{-1}b_{g+1-k}c_{g+1-k}=1$ ($1 \leq k \leq r$) and $c_r=1$.

On the other hand, by comparing the equality (\ref{c}) for $i = g+1-k$ with that for $i=g-k$ we obtain $c_{g+1-k}=b_{g+1-k}^{-1}c_{g-k}(a_{g+1-k}b_{g+1-k}a_{g+1-k}^{-1})$. 
Substituting this to $b_kc_{g-k}^{-1}b_{g+1-k}c_{g+1-k}=1$ we have $b_ka_{g+1-k}b_{g+1-k}a_{g+1-k}^{-1}=1$ ($1 \leq k \leq r$). 
We can track back the above argument conversely to show that $G_1$ has a presentation with generators $a_1,b_1,a_2,b_2,\ldots,a_g,b_g$ and with relations 
\begin{align*}
&c_g=c_r=1; \\
&a_k=a_{g+1-k}^{-1} \ \ \mathrm{and} \ \ b_k=a_{g+1-k}b_{g+1-k}^{-1}a_{g+1-k}^{-1}, \ \ 1\leq k\leq r; \\ 
&a_{g+1-k}=a_k^{-1} \ \ \mathrm{and} \ \ b_{g+1-k}=a_{g+1-k}^{-1}b_k^{-1}a_{g+1-k}, \ \ 1\leq k\leq r. 
\end{align*}
It turns out that this presentation is equivalent to the presentation with generators $a_1, b_1, a_2, b_2, \cdots, a_r, b_r$ and with relation $c_r=1$, namely, $G_1=\pi_1(\Sigma_r)$.

Now we suppose that $g=2r+1$. 
In this time we have $a_{r+1}=a_{r+1}^\prime=1$ since they belong to $\mathcal{B}_{1}^g$; hence $a_{r+1}^\prime=c_ra_{r+1}$ gives $c_r=1$. 
Having this a parallel argument as in the case of $g=2r$ shows that $G_1$ is isomorphic to $\langle a_1,b_1,a_2,b_2,\ldots,a_r,b_r \mid c_r=1 \rangle=\pi_1(\Sigma_r)$.

Next, we compute $G_2$. 
If $[g/2]=2t$, then by a similar argument as in the proof of $G_1=\pi_1(\Sigma_{[g/2]})$, we see that 
$G_2$ has a presentation with generators $a_1,b_1,a_2,b_2,\ldots,a_{2t},b_{2t}$ and with relations 
\begin{align*}
&c_{2t}=c_t=1; \\
&a_k=a_{[g/2]+1-k}^{-1} \ \ \mathrm{and} \ \  b_k=a_{[g/2]+1-k}b_{[g/2]+1-k}^{-1}a_{[g/2]+1-k}^{-1}, \ \ 1\leq k\leq t, 
\\ & a_{[g/2]+1-k} = a_k^{-1} \ \ \mathrm{and} \ \ b_{[g/2]+1-k} =a_{[g/2]+1-k}^{-1} b_k^{-1} a_{[g/2]+1-k}, \ \ 1\leq k\leq t, 
\end{align*}
i.e., $G_2=\pi_1(\Sigma_t)$. 
Similarly, if $[g/2]=2t+1$, then $G_2=\pi_1(\Sigma_t)$.

Finally, we compute $G_3$. 
Suppose that $[g/2]=2t$. 
It is easy to check that, up to conjugation, the following equalities hold in $\pi_1(\Sigma_g)$: 
\begin{align*}
&\phi_n(B_{0,s}^{2t})=a_t^na_{t-1} \cdots a_2a_1 B_{0,s}^{2t}; \\
&\phi_n(B_{2k-1}^{2t})=b_{2t+1-k}^{-1}a_t^na_{t+1}\cdots a_{k+1}a_k B_{2k-1}^{2t}, \ 1\leq k\leq t-1; \\
&\phi_n(B_{2k}^{2t})=b_{2t+1-k}^{-1}a_t^na_{t+1}\cdots a_{k+2}a_{k+1} B_{2k}^{2t}, \ 1\leq k\leq t-1; \\
&\phi_n(B_{2k-1}^{2t})=a_t^n B_{2t-1}^{2t}; \ \ \phi_n(B_{2t}^{2t})=B_{2t}^{2t}; \ \ \phi_n(c_t)=c_t.
\end{align*}
Noting that they vanish in $G_3$, a simple computation gives relations $a_1=a_2=\cdots=a_{t-1}=a_t^n=b_{t+2}=b_{t+3}=\cdots=b_{2t}=1$. 
Therefore, $G_3$ has a presentation with generators $a_1,b_1,\ldots,a_{2t},b_{2t}$ and with relations 
\begin{align*}
&a_k=a_{[g/2]+1-k}^{-1} \ \ \mathrm{and} \ \ 
b_k=a_{[g/2]+1-k}b_{[g/2]+1-k}^{-1}a_{[g/2]+1-k}^{-1}, \ \ 1\leq k\leq t; \\
& a_{[g/2]+1-k} = a_k^{-1} \ \ \mathrm{and} \ \ b_{[g/2]+1-k} =a_{[g/2]+1-k}^{-1} b_k^{-1} a_{[g/2]+1-k}, \ \ 1\leq k\leq t;\\ 
&a_1=a_2=\cdots=a_{t-1}=a_t^n=b_{t+2}=b_{t+3}=\cdots=b_{2t}=c_{2t}=c_t=1.
\end{align*}
This presentation is equivalent to the presentation with generators $a_t,b_t$ and with relations $a_t^n=a_tb_ta_t^{-1}b_t^{-1}=1$.

For $[g/2]=2t+1$, a similar computation gives relations $a_1=a_2=\cdots=a_{t-1}=a_t^n=b_{t+3}=b_{t+4}=\cdots=b_{2t+1}=1$. 
Hence, we see that $G_3$ has a presentation with generators $a_t,b_t$ and with relations $a_t^n=a_tb_ta_t^{-1}b_t^{-1}=1$.

This completes the proof. 
\end{proof}

\end{document}